\documentclass[11pt,reqno]{amsart}

\usepackage[T1]{fontenc}
\usepackage[utf8]{inputenc}
\usepackage{lmodern}
\usepackage{microtype}

\usepackage{amsmath,amssymb,amsthm}
\usepackage{xcolor}
\usepackage{cite}
\usepackage[
  colorlinks=true,
  citecolor=blue,
  linkcolor=blue,
  urlcolor=blue
]{hyperref}

\numberwithin{equation}{section}
\newtheorem{theorem}{Theorem}[section]
\newtheorem{proposition}[theorem]{Proposition}
\newtheorem{lemma}[theorem]{Lemma}

\title[Dihedral Rigidity by Smooth Approximation]
{Dihedral Rigidity for Convex Polytopes by Smooth Approximation}
\author[Y. Bi]{Yuchen Bi}
\address[Yuchen Bi]{Mathematical Institute, Department of Pure Mathematics,
University of Freiburg, Ernst-Zermelo-Stra{\ss}e 1,
D-79104 Freiburg im Breisgau, Germany}
\email{yuchen.bi@math.uni-freiburg.de}

\hypersetup{
  pdftitle={Dihedral Rigidity for Convex Polytopes by Smooth Approximation},
  pdfauthor={Yuchen Bi}
}

\begin{document}

\begin{abstract}
Following Brendle's smooth approximation approach, we give a proof of Gromov's dihedral rigidity conjecture for convex polytopes. 
\end{abstract}

\maketitle

\section{Introduction}
Gromov's dihedral rigidity conjecture~\cite[Section 2.2]{GromovCorners} states that a smooth metric on a compact convex polytope is flat if its scalar curvature and the mean curvatures of its faces are nonnegative and its interior dihedral angles are no larger than the corresponding Euclidean angles.

    Using minimal hypersurface descent, Li~\cite{Li,LiCorrection,LiPrism} proved the conjecture for several classes of polyhedra of cone or prism type. An index-theoretic approach to the general conjecture for manifolds with polyhedral boundary was developed by Wang, Xie, and Yu~\cite{WangXieYu} and further pursued by Wang and Xie~\cite{WangXie}; a self-contained proof in dimension three was recently given by Wang, Xie, and Yu~\cite{WangXieYuThree}. Brendle~\cite{Brendle} proved the matching-angle case using smooth inner approximation, and Brendle and Wang~\cite{BrendleWang} extended this method to polytopes with acute dihedral angles.

In this paper, we give a short proof of Gromov's dihedral rigidity conjecture based on the smooth approximation method of Brendle~\cite{Brendle} and Brendle--Wang~\cite{BrendleWang}. We begin by fixing notation and conventions. Throughout, gradients of functions, as well as norms and inner products on \(\mathbb R^n\), are Euclidean unless marked by \(g\). We write \(d\) and \(d_g\) for the Euclidean and \(g\)-distances, respectively, and \(d_{\mathrm H}^g\) for the Hausdorff distance induced by \(d_g\).

Let
\[
P=\bigcap_{a=1}^m\{u_a\leq 0\}\subset\mathbb R^n
\]
be a compact convex polytope with nonempty interior. We assume that this representation is irredundant, meaning that none of the inequalities \(u_a\leq 0\) can be omitted. Normalize the affine functions \(u_a\) by \(|\nabla u_a|=1\), and write
\[
F_a=P\cap\{u_a=0\},\qquad N_a=\nabla u_a.
\]
For a face \(G\) of \(P\), let \(G^\circ\) denote its relative interior. Let \(g\) be a smooth metric on a neighborhood of \(P\), and put
\[
\nu_a=\frac{\nabla^g u_a}{|du_a|_g}.
\]
Along \(F_a\), this is the outward unit normal with respect to \(g\). If \(F_a\) and \(F_b\) meet in a codimension-two face, define along that face
\[
\alpha_{ab} =\arccos\langle N_a,N_b\rangle,\qquad\alpha_{ab}^g=\arccos\langle\nu_a,\nu_b\rangle_g.
\]
These are the exterior dihedral angles; the corresponding interior dihedral angles are \(\pi-\alpha_{ab}\) and \(\pi-\alpha_{ab}^g\), respectively. We write \(R_g\) for the scalar curvature of \(g\) and \(H_{F_a}^g\) for the mean curvature of \(F_a\) with respect to \(\nu_a\). Mean curvature is computed with the convention that the Euclidean unit sphere has positive mean curvature.

\begin{theorem}\label{thm:main}
Let \(n\geq 3\). Suppose that \(R_g\geq 0\) on \(P\), that \(H_{F_a}^g\geq 0\) on each \(F_a\), and that \(\alpha_{ab}^g\geq\alpha_{ab}\) along every codimension-two face. Then \(g\) is flat and every \(F_a\) is totally geodesic. Moreover, whenever \(a\neq b\) and \(F_a\cap F_b\neq\emptyset\),
\[
\langle\nu_a,\nu_b\rangle_g=\langle N_a,N_b\rangle\qquad\text{on }F_a\cap F_b.
\]
\end{theorem}

\section{Smooth inner approximations and approximate normals}
In this section, we construct smooth inner approximations \(P_\lambda\) of \(P\) and associated sphere-valued maps
\[
\eta_\lambda:\partial P_\lambda\longrightarrow\mathbb S^{n-1}
\]
that approximate the Euclidean outward unit normal fields of \(\partial P_\lambda\).

\subsection{Construction of \texorpdfstring{\(P_\lambda\)}{P\_lambda}}

We use the smooth approximation of Brendle~\cite{Brendle} and Brendle--Wang~\cite{BrendleWang}, recalled here in a form adapted to our argument. Fix a smooth convex function
\(\Phi:\mathbb R\to[0,\infty)\) such that
\[
\Phi=0\quad\text{on }(-\infty,-2],\qquad\Phi'>0\quad\text{on }(-2,\infty),\qquad\Phi(0)=1.
\]
Define
\[
\mathcal F_\lambda=\sum_{a=1}^m\Phi(\lambda u_a),\qquad P_\lambda=\{\mathcal F_\lambda\leq 1\},\qquad \Sigma_\lambda=\partial P_\lambda.
\]
For \(x\in P\), set
\[
I_\lambda(x)=\{a:u_a(x)>-2/\lambda\}.
\]

We first record two estimates.
\begin{lemma}\label{lem:rounding-estimates}
There exist constants \(C,c>0\), depending only on \(P\) and the metric \(g\), with the following properties. For every nonempty \(I\subset\{1,\ldots,m\}\) such that
\[
F_I:=\bigcap_{a\in I}F_a\neq\emptyset,
\]
one has
\begin{equation}\label{eq:face-distance}
d_g(x,F_I)\leq C\max_{a\in I}|u_a(x)|\qquad (x\in P).
\end{equation}
Moreover, for every sufficiently large \(\lambda\), every \(x\in P\), and every family \((t_a)_{a\in I_\lambda(x)}\) of nonnegative numbers,
\begin{equation}\label{eq:normal-combination}
\left|\sum_{a\in I_\lambda(x)}t_a\,du_a(x)\right|_g\geq c\sum_{a\in I_\lambda(x)}t_a|du_a(x)|_g.
\end{equation}
\end{lemma}

\begin{proof}
For the first estimate, fix \(I\), choose \(y\in F_I\) with \(d(x,y)=d(x,F_I)\), and set \(v=x-y\) and \(J=\{j:u_j(y)=0\}\). Since \(y\) is the nearest point to \(x\) in \(F_I\),
\[
v\in N_{F_I}(y):=\{\xi:\langle\xi,z-y\rangle\leq0\text{ for all }z\in F_I\}.
\]
At \(y\), the set \(F_I\) is defined by \(u_a=0\), \(a\in I\), and the inequalities \(u_j\leq0\), \(j\in J\setminus I\). Hence
\[
C_{I,J}:=\operatorname{cone}\{\nabla u_j:j\in J\setminus I\}+\operatorname{span}\{\nabla u_a:a\in I\}=N_{F_I}(y).
\]
Moreover, \(du_j(v)=u_j(x)\leq0\) for \(j\in J\). We claim that, uniformly over all possible pairs \(I\subset J\),
\[
\max_{a\in I}|du_a(w)|\geq c_0|w|
\]
whenever \(w\in C_{I,J}\) and \(du_j(w)\leq0\) for \(j\in J\setminus I\). Otherwise, by compactness, there would be a fixed pair \(I\subset J\) and a unit vector \(w\in C_{I,J}\) such that
\[
du_a(w)=0\quad(a\in I),\qquad du_j(w)\leq0\quad(j\in J\setminus I).
\]
Writing
\[
w=\sum_{j\in J\setminus I}\mu_j\nabla u_j+\sum_{a\in I}\lambda_a\nabla u_a,\qquad \mu_j\geq0,
\]
and taking the inner product with \(w\) gives \(1=|w|^2\leq0\), a contradiction. Applying the claim to \(v\), we obtain
\[
d(x,F_I)=|v|\leq c_0^{-1}\max_{a\in I}|u_a(x)|.
\]
The equivalence of \(g\) and the Euclidean metric on \(P\) yields~\eqref{eq:face-distance}.

For the second estimate, fix \(p_0\in P^\circ\) and set \(\delta:=\min_a\bigl(-u_a(p_0)\bigr)>0\). If \(\lambda\geq 4/\delta\), then every \(a\in I_\lambda(x)\) satisfies
\[
du_a(x-p_0)=u_a(x)-u_a(p_0)\geq \delta-\frac{2}{\lambda}\geq \frac{\delta}{2}.
\]
Consequently,
\[
\left|\sum_{a\in I_\lambda(x)}t_a\,du_a\right||x-p_0|\geq\frac{\delta}{2}\sum_{a\in I_\lambda(x)}t_a.
\]
Since \(|x-p_0|\) is uniformly bounded on \(P\), this proves~\eqref{eq:normal-combination}.
\end{proof}

We now describe the geometry of \(P_\lambda\) for large \(\lambda\). Set
\[
E_P:=\bigcup_{|I|=2}F_I,\qquad K_P:=\bigcup_{|I|=3}F_I,
\]
where \(I\subset\{1,\ldots,m\}\). From now on, we fix a \(p_0\in P^\circ\).

\begin{lemma}\label{lem:smoothing}
For all sufficiently large \(\lambda\), \(P_\lambda\) is a compact convex domain with smooth boundary \(\Sigma_\lambda\), contained in \(P\). Every segment joining \(p_0\) to a point \(y\in\partial P\) meets \(\Sigma_\lambda\) in exactly one point. The resulting map
\[
\Pi_\lambda:\partial P\longrightarrow\Sigma_\lambda
\]
is bi-Lipschitz with constants independent of \(\lambda\), and
\[
\sup_{y\in\partial P}d_g\bigl(\Pi_\lambda(y),y\bigr)\leq C\lambda^{-1}.
\]
In particular, \(d_{\mathrm H}^g(\Sigma_\lambda,\partial P)\leq C\lambda^{-1}\). Moreover, for every \(x\in\Sigma_\lambda\),
\[
F_{I_\lambda(x)}\neq\emptyset,\qquad d_g\bigl(x,F_{I_\lambda(x)}\bigr)\leq C\lambda^{-1}.
\]
Consequently,
\[
|I_\lambda(x)|\geq2\Longrightarrow d_g(x,E_P)\leq C\lambda^{-1},\qquad|I_\lambda(x)|\geq3\Longrightarrow d_g(x,K_P)\leq C\lambda^{-1}.
\]
Finally, for every \(a\) and every \(V\Subset F_a^\circ\), there exists \(\lambda_V>0\) such that, for all \(\lambda\geq\lambda_V\), \(\Sigma_\lambda\) agrees with \(F_a\) in a fixed neighborhood of \(\overline V\).
\end{lemma}

\begin{proof}
For all sufficiently large \(\lambda\), \(\mathcal F_\lambda(p_0)=0\). Since \(\mathcal F_\lambda\) is convex and \(\mathcal F_\lambda(x)>1\) whenever \(u_a(x)>0\) for some \(a\),
\(P_\lambda\) is a compact convex domain contained in \(P\).

Values below \(-2\) contribute neither to \(\sum_a\Phi(s_a)\) nor to \(\sum_a\Phi'(s_a)\). The continuous function \(\sum_a\Phi'(s_a)\) is positive on the compact set
\[
\left\{(s_1,\ldots,s_m)\in[-2,0]^m:\sum_a\Phi(s_a)=1\right\}.
\]
Applying \eqref{eq:normal-combination} with \(t_a=\Phi'(\lambda u_a)\), we obtain
\begin{equation}\label{eq:gradient-lower-bound}
|d\mathcal F_\lambda|_g=\lambda\left|\sum_a\Phi'(\lambda u_a)\,du_a\right|_g\geq c\lambda\sum_a\Phi'(\lambda u_a)|du_a|_g\geq c\lambda
\end{equation}
on \(\Sigma_\lambda\). Thus \(\Sigma_\lambda\) is smooth.

Set \(\delta:=\min_a(-u_a(p_0))>0\). If \(|x-p_0|\leq\delta/2\), then \(u_a(x)\leq u_a(p_0)+|x-p_0|\leq-\delta/2\) for every \(a\).
Hence, for large \(\lambda\),
\[
B_{\delta/2}(p_0)\subset P_\lambda\subset P\subset B_R(p_0)
\]
for some \(R>0\). Convex domains lying between two fixed concentric balls have uniformly bi-Lipschitz radial parametrizations. It follows that every segment joining \(p_0\) to \(y\in\partial P\) meets \(\Sigma_\lambda\) in exactly one point and that the resulting map \(\Pi_\lambda:\partial P\to\Sigma_\lambda\) is uniformly bi-Lipschitz.

Choose \(A>2/\delta\) and set \(z=(1-A/\lambda)y+(A/\lambda)p_0\). For \(y\in\partial P\), one has \(u_a(z)\leq-A\delta/\lambda<-2/\lambda\) for every \(a\), so \(z\in P_\lambda\) and \(\Pi_\lambda(y)\in[y,z]\). Therefore
\[
d_g\bigl(\Pi_\lambda(y),y\bigr)\leq C\lambda^{-1}.
\]
The Hausdorff estimate follows from the surjectivity of \(\Pi_\lambda\).

Since \(P\) is compact and there are only finitely many \(J\subset\{1,\ldots, m\}\), there is \(\varepsilon>0\) such that \(F_J=\emptyset\) implies \(\max_{a\in J}|u_a|\geq\varepsilon\) on \(P\). Fix \(x\in \Sigma_\lambda\) and set \(I=I_\lambda(x)\). Since \(\mathcal F_\lambda(x)=1\), the set \(I\) is nonempty and \(\max_{a\in I}|u_a(x)|<2/\lambda\). Thus \(F_I\neq\emptyset\) for all sufficiently large \(\lambda\), and \eqref{eq:face-distance} gives
\[
d_g(x,F_I)\leq C\lambda^{-1}.
\]
If \(|I|\geq2\), then \(F_I\subset E_P\); if \(|I|\geq3\), then \(F_I\subset K_P\).

Finally, if \(V\Subset F_a^\circ\), there are a neighborhood \(U\) of \(\overline V\) and \(\varepsilon_V>0\) such that \(u_b\leq-\varepsilon_V\) on \(U\) for every \(b\neq a\). Choose \(\lambda_V\) so that \(\lambda_V\varepsilon_V\geq2\). Then, for every \(\lambda\geq\lambda_V\), one has \(\mathcal F_\lambda=\Phi(\lambda u_a)\) on \(U\), and hence
\[
\Sigma_\lambda\cap U=F_a\cap U.
\]
\end{proof}

\subsection{Normals near codimension-two faces}
\label{subsec:codimension-two-normals}

We examine the outward unit normal and the second fundamental form of \(\Sigma_\lambda\) near a codimension-two face. The calculations below adapt arguments of Brendle and Wang~\cite{BrendleWang}.

On \(\Sigma_\lambda\), set \(\beta_j=\Phi^\prime(\lambda u_j)|du_j|_g\). Since \(\nabla^g\mathcal F_\lambda=\lambda\sum_j\beta_j\nu_j\), the outward unit normal to \(\Sigma_\lambda\) is
\[
\nu_\lambda=\frac{\sum_j\beta_j\nu_j}{\left|\sum_j\beta_j\nu_j\right|_g}.
\]
By the defining properties of \(\Phi\), one has \(\beta_j(x)>0\) precisely when \(j\in I_\lambda(x)\).

Fix \(C_0>0\), to be chosen sufficiently large below. Throughout this subsection, we consider points \(x\in\Sigma_\lambda\) satisfying
\[
I_\lambda(x)=\{a,b\},\qquad d_g(x,K_P)>C_0\lambda^{-1}.
\]
Lemma~\ref{lem:smoothing} gives
\[
F_a\cap F_b\neq\emptyset,\qquad d_g(x,F_a\cap F_b)\leq C\lambda^{-1}.
\]
Since every face of codimension at least three is contained in \(K_P\), the assumption \(d_g(x,K_P)>C_0\lambda^{-1}\) implies for \(C_0\) sufficiently large that \(F_a\cap F_b\) has codimension two.

Abbreviate \(\Sigma=\Sigma_\lambda\) and \(\nu=\nu_\lambda\), and set \(L=\ker du_a\cap\ker du_b\). Since \(F_a\cap F_b\) has codimension two, \(du_a\) and \(du_b\) are linearly independent, and hence \(\dim L=n-2\). Since only the \(a\)- and \(b\)-terms contribute to
\(d\mathcal F_\lambda\) at \(x\), one has \(L\subset\ker d\mathcal F_\lambda=T_x\Sigma\). Thus the \(g\)-orthogonal complement of \(L\) in \(T_x\Sigma\) is a line. Let \(e\) be a \(g\)-unit vector spanning this line. The relation \(d\mathcal F_\lambda(e)=0\) gives
\[
\Phi^\prime(\lambda u_a)du_a(e)+\Phi^\prime(\lambda u_b)du_b(e)=0.
\]
Since \(e\notin L\), this relation shows that \(du_a(e)\) and \(du_b(e)\) are nonzero and have opposite signs. After replacing \(e\) by \(-e\) if necessary, we may assume that \(du_b(e)>0\). Set
\[
\kappa=\frac{\lambda\sum_{j\in\{a,b\}}\Phi^{\prime\prime}(\lambda u_j)du_j(e)^2}{|\beta_a\nu_a+\beta_b\nu_b|_g}.
\]
At \(x\), \eqref{eq:gradient-lower-bound} gives
\[
|\beta_a\nu_a+\beta_b\nu_b|_g=\lambda^{-1}|d\mathcal F_\lambda|_g\geq c.
\]
Since \(\Phi''(\lambda u_j)\) and \(du_j(e)\) are uniformly bounded, it follows that
\begin{equation}\label{eq:kappa-bound}
0\leq\kappa\leq C\lambda.
\end{equation}
By the level-set formula, the second fundamental form \(\mathrm{II}_{\Sigma}^{g}\) of \(\Sigma\) with respect to \(\nu\)
is given by \(\mathrm{II}_{\Sigma}^{g}=|d\mathcal F_\lambda|_g^{-1}\nabla_g^2\mathcal F_\lambda\) on \(T_x\Sigma\). At \(x\), \(|d\mathcal F_\lambda|_g=\lambda|\beta_a\nu_a+\beta_b\nu_b|_g\). Every \(X\in T_x\Sigma\) can be written as
\(X=X_L+\langle X,e\rangle_g e\), where \(X_L\in L\). Since \(du_a\) and \(du_b\) vanish on \(L\), we have \(du_j(X)=du_j(e)\langle X,e\rangle_g\) for \(j\in\{a,b\}\). Expanding \(\nabla_g^2\mathcal F_\lambda\) gives
\begin{equation}\label{eq:second-form-decomposition}
\mathrm{II}_{\Sigma}^{g}(X,Y)=\frac{\sum_{j\in\{a,b\}}\Phi^\prime(\lambda u_j)\nabla_g^2u_j(X,Y)
}{|\beta_a\nu_a+\beta_b\nu_b|_g}+\kappa\langle X,e\rangle_g\langle Y,e\rangle_g
\end{equation}
for \(X,Y\in T_x\Sigma\). The first term on the right-hand side is uniformly bounded. Consequently,
\begin{equation}\label{eq:second-form-estimates}
\mathrm{II}_\Sigma^g(e,e)=\kappa+O(1),\qquad H_\Sigma^g=\kappa+O(1),\qquad\bigl|\mathrm{II}_\Sigma^g(Y,\cdot)\bigr|_g\leq C
\end{equation}
for every \(g\)-unit vector \(Y\in L\).

We now describe how the normal turns in the direction transverse to \(F_a\cap F_b\), taking the two face normals at a nearby point of \(F_a\cap F_b\) as reference. Let \(x_{ab}\) be the Euclidean projection of \(x\) onto \(\{u_a=u_b=0\}\). Since \(F_a\cap F_b\) is contained in this affine space,
\[
|x-x_{ab}|\leq d(x,F_a\cap F_b)\leq C\lambda^{-1}.
\]
If \(x_{ab}\notin F_a\cap F_b\), the nearest point to \(x\) in \(F_a\cap F_b\) lies in its relative boundary and hence in \(K_P\).
This would give \(d_g(x,K_P)\leq C\lambda^{-1}\), a contradiction for \(C_0\) sufficiently large. Thus \(x_{ab}\in F_a\cap F_b\).

Let \(\bar\nu_a\) and \(\bar\nu_b\) be obtained by \(g\)-parallel transporting \(\nu_a(x_{ab})\) and \(\nu_b(x_{ab})\) to \(x\) along the straight segment from \(x_{ab}\) to \(x\). Put \(\alpha=\arccos\langle\bar{\nu}_a,\bar{\nu}_b\rangle_g=\alpha_{ab}^g(x_{ab})\). Thus \(\alpha_{ab}\leq\alpha\) by assumption. Moreover, \(\alpha\) stays uniformly away from \(0\) and \(\pi\) by compactness. Set
\[
\bar{Z}=\beta_a\bar{\nu}_a+\beta_b\bar{\nu}_b,\qquad \bar{\nu}=\frac{\bar{Z}}{|\bar{Z}|_g}.
\]
The vector \(\bar{\nu}\) lies on the minimizing geodesic in the \(g_x\)-unit sphere from \(\bar{\nu}_a\) to \(\bar{\nu}_b\). Let \(s\in[0,\alpha]\) be the spherical distance from \(\bar{\nu}_a\) to \(\bar{\nu}\).
The normal \(\nu\) lies on the minimizing geodesic in the \(g_x\)-unit sphere from \(\nu_a\) to \(\nu_b\). Its unit tangent at \(\nu\), oriented toward \(\nu_b\), is \(e\), since it lies in \(T_x\Sigma\cap L^{\perp_g}\) and has positive \(du_b\)-value. The following lemma compares the two spherical arcs.

\begin{lemma}\label{lem:s-estimates}
For all sufficiently large \(\lambda\) and every \(x\in\Sigma_\lambda\) with \(I_\lambda(x)=\{a,b\}\) and \(d_g(x,K_P)\geq C_0\lambda^{-1}\), one has
\[
ds(e)=\kappa+O(1),\qquad |d\alpha|_g\leq C.
\]
Moreover, \(|ds(Y)|\leq C\) for every \(g\)-unit vector \(Y\in L\).
\end{lemma}

\begin{proof}
Since \(x_{ab}\) depends smoothly on \(x\), differentiating the equation for parallel transport along the straight segment joining \(x_{ab}\) to \(x\) gives
\[
|\bar\nu_j-\nu_j|_g\leq C\lambda^{-1}\quad (j\in\{a,b\}),\qquad|\nabla^g\bar\nu_a|_g+|\nabla^g\bar\nu_b|_g\leq C.
\]
Since \(\alpha\) stays uniformly away from \(0\) and \(\pi\), differentiating \(\cos\alpha=\langle\bar\nu_a,\bar\nu_b\rangle_g\) gives \(|d\alpha|_g\leq C\).

We next compare \(\nu\) and \(\bar\nu\). Set \(Z=\beta_a\nu_a+\beta_b\nu_b\). At \(x\), \eqref{eq:gradient-lower-bound} gives \(|Z|_g\geq c\), while the
definition of \(\beta_j\) gives \(\beta_a+\beta_b\leq C\). Since \(Z-\bar Z=\sum_{j\in\{a,b\}}\beta_j(\nu_j-\bar\nu_j)\), it follows that \(|Z-\bar Z|_g\leq C\lambda^{-1}\) and \(|\bar Z|_g\geq c\) for all sufficiently large \(\lambda\).

Moreover, \(d\beta_j=\lambda\Phi''(\lambda u_j)|du_j|_g\,du_j+\Phi'(\lambda u_j)d|du_j|_g\), and hence \(|d\beta_j|_g\leq C\lambda\). Differentiating the expression for \(Z-\bar Z\) gives
\[
\nabla^g(Z-\bar Z)=\sum_{j\in\{a,b\}}d\beta_j\otimes(\nu_j-\bar\nu_j)+\sum_{j\in\{a,b\}}\beta_j\nabla^g(\nu_j-\bar\nu_j).
\]
Consequently,
\[
\bigl|\nabla^g(Z-\bar Z)\bigr|_g\leq C,\qquad|\nabla^gZ|_g+|\nabla^g\bar Z|_g\leq C\lambda.
\]
Differentiating \(\nu=Z/|Z|_g\) and \(\bar\nu=\bar Z/|\bar Z|_g\), we obtain
\begin{equation}\label{eq:normal-comparison}
|\nu-\bar\nu|_g\leq C\lambda^{-1},\qquad\bigl|\nabla^g\nu-\nabla^g\bar\nu\bigr|_g\leq C.
\end{equation}

Set \(T_a=(\bar\nu_b-\cos\alpha\,\bar\nu_a)/\sin\alpha\). Then
\[
\bar\nu=\cos s\,\bar\nu_a+\sin s\,T_a,\qquad T=-\sin s\,\bar\nu_a+\cos s\,T_a,
\]
where \(T\) is the unit tangent at \(\bar\nu\) oriented toward \(\bar\nu_b\). The bounds for \(\nabla^g\bar\nu_a\), \(\nabla^g\bar\nu_b\), and \(d\alpha\) give \(|\nabla^gT_a|_g\leq C\). Differentiating the first identity and taking the inner product with \(T\), we obtain
\begin{equation}\label{eq:ds-normal}
ds(X)=\langle\nabla_X^g\bar\nu,T\rangle_g+O(|X|_g) \qquad (X\in T_x\Sigma).
\end{equation}

Since the endpoint angles stay in a compact subset of \((0,\pi)\), the oriented unit tangent to the minimizing spherical geodesic depends smoothly on its endpoints and the point on the geodesic. The bounds for \(\bar\nu_a-\nu_a\), \(\bar\nu_b-\nu_b\), and \eqref{eq:normal-comparison} therefore give \(|T-e|_g\leq C\lambda^{-1}\). On the other hand, \eqref{eq:kappa-bound} and \eqref{eq:second-form-decomposition} give
\[
|\nabla_e^g\nu|_g=\bigl|\mathrm{II}_\Sigma^g(e,\cdot)\bigr|_g\leq\kappa+C\leq C\lambda.
\]
Together with \eqref{eq:normal-comparison}, these estimates imply
\[
\langle\nabla_e^g\bar\nu,T\rangle_g=\langle\nabla_e^g\nu,e\rangle_g+O(1).
\]
It now follows from \eqref{eq:ds-normal} and \eqref{eq:second-form-estimates} that
\[
ds(e)=\langle\nabla_e^g\nu,e\rangle_g+O(1)=\mathrm{II}_\Sigma^g(e,e)+O(1)=\kappa+O(1).
\]

Finally, if \(Y\in L\) is a \(g\)-unit vector, then \eqref{eq:normal-comparison}, \eqref{eq:ds-normal}, and \eqref{eq:second-form-estimates} give
\[
|ds(Y)|\leq|\nabla_Y^g\nu|_g+C=\bigl|\mathrm{II}_\Sigma^g(Y,\cdot)\bigr|_g+C\leq C.
\]
\end{proof}

\subsection{The map \texorpdfstring{\(\eta_\lambda\)}{eta\_lambda}}

We first define a sphere-valued map away from \(K_P\), where at most two indices belong to \(I_\lambda(x)\). For \(U,V\in \mathbb S^{n-1}\) with \(U\neq -V\), let \(\Gamma(U,V,\cdot)\) denote the unique minimizing geodesic from \(U\) to \(V\), parameterized on \([0,1]\) at constant speed. Writing \(\theta=\arccos\langle U,V\rangle\), one has
\begin{equation}\label{eq:spherical-interpolation}
\Gamma(U,V,t)=\frac{\sin((1-t)\theta)}{\sin\theta}U+\frac{\sin(t\theta)}{\sin\theta}V,
\end{equation}
where the right-hand side is understood by smooth extension at \(\theta=0\).

By Lemma~\ref{lem:smoothing}, after increasing \(C_0\) if necessary,
\[
x\in\Sigma_\lambda,\qquad d_g(x,K_P)>C_0\lambda^{-1}\quad\Longrightarrow\qquad |I_\lambda(x)|\leq 2.
\]
On this open subset of \(\Sigma_\lambda\), define
\begin{equation}\label{eq:zeta-definition}
\zeta_\lambda(x)=\begin{cases}
    N_a, &I_\lambda(x)=\{a\}, \\[2mm]
\Gamma\left(N_a, N_b, \frac{s}{\alpha}\right), &I_\lambda(x)=\{a,b\},  
\end{cases}
\end{equation}
where \(s\) and \(\alpha=\alpha_{ab}^g(x_{ab})\) are as in Subsection~\ref{subsec:codimension-two-normals}.  The second line is independent of the order of \(a\) and \(b\). Interchanging them replaces \(s/\alpha\) by \(1-s/\alpha\), and
\[
\Gamma(N_b,N_a,1-s/\alpha)
=
\Gamma(N_a,N_b,s/\alpha).
\]
\begin{lemma}\label{lem:zeta-smooth}
After increasing \(C_0\) if necessary, the map \(\zeta_\lambda\) is well-defined and smooth on 
\[
\{x\in\Sigma_\lambda: d_g(x,K_P)>C_0\lambda^{-1}\}
\]
for all sufficiently large \(\lambda\).
\end{lemma}
\begin{proof}
    Fix \(x_0\in\Sigma_\lambda\) with \(d_g(x_0,K_P)>C_0\lambda^{-1}\), and set
    \[
    J_0=\{j:u_j(x_0)\geq -2/\lambda\}\supset I_\lambda(x_0).
    \]
    As in the proof of Lemma~\ref{lem:smoothing}, there is \(\varepsilon>0\) such that
    \[
    F_J=\emptyset\quad\Longrightarrow\quad\max_{j\in J}|u_j(x)|\geq\varepsilon\qquad (x\in P).
    \]
    Since \(\max_{j\in J_0}|u_j(x_0)|\leq 2\lambda^{-1}\), we have \(F_{J_0}\neq\emptyset\) for all sufficiently large \(\lambda\).

    If \(|J_0|\geq 3\), then \(F_{J_0}\subset K_P\), and \eqref{eq:face-distance} gives
    \[
    d_g(x_0,K_P)\leq d_g(x_0,F_{J_0})\leq C\max_{j\in J_0}|u_j(x_0)|\leq 2C\lambda^{-1},
    \]
contrary to the choice of \(C_0\) once \(C_0>2C\). Thus \(|J_0|\leq 2\). Moreover, if \(J_0=\{a,b\}\), the same argument shows that \(F_a\cap F_b\) has codimension two.

For every \(j\notin J_0\), one has \(u_j(x_0)< -2/\lambda\). Hence, after restricting to a neighborhood \(U\) of \(x_0\),
\[
I_\lambda(x)\subset J_0 \qquad (x\in U).
\]
If \(|J_0|=1\), then \(\zeta_\lambda\) is constant on \(U\). If \(J_0=\{a,b\}\), then \(I_\lambda\subset\{a,b\}\) on \(U\). Taking the \(g\)-inner product of \(\bar{Z}=\beta_a\bar{\nu}_a+\beta_b\bar{\nu}_b\) with \(\bar{\nu}_a\) gives \(\cos s=(\beta_a+\beta_b\cos\alpha)/|\bar{Z}|_g\). The component of \(\bar{Z}\) orthogonal to \(\bar{\nu}_a\) is \(\beta_b(\bar{\nu}_b-\cos\alpha\bar{\nu}_a)\), whose norm is \(\beta_b\sin\alpha\). Hence \(\sin s=\beta_b\sin\alpha/|\bar{Z}|_g\). When \(\beta_a,\beta_b>0\), the vector \(\bar{Z}\) lies strictly between \(\bar{\nu}_a\) and \(\bar{\nu}_b\), so \(0<s<\alpha<\pi\). It follows that 
\[
s=\arccos\left(\frac{\beta_a+\beta_b\cos\alpha}{|\bar{Z}|_g}\right)
\]
is smooth. Since \(t\mapsto\Gamma(N_a,N_b,t)\) is smooth, \(\zeta_\lambda\) is smooth where \(\beta_a,\beta_b>0\).

It remains to consider points where one of \(\beta_a,\beta_b\) vanishes. Suppose that \(\beta_b=0\), so \(u_b(x_0)=-2/\lambda\) . Since \(\mathcal F_\lambda=1\) and \(I_\lambda(x)\subset\{a,b\}\), one has \(\Phi(\lambda u_a)=1\), and therefore \(u_a=0\) and \(\beta_a=\Phi'(0)|du_a|_g>0\). After shrinking \(U\), we may assume that \(\beta_a+\beta_b\cos\alpha>0\) on \(U\). The preceding discussions give
\[
s=\arctan\left(\frac{\beta_b\sin\alpha}{\beta_a+\beta_b\cos\alpha}\right).
\]
The right-hand side is smooth on \(U\) and vanishes when \(\beta_b=0\). Hence \(\Gamma(N_a,N_b,s/\alpha)\) agrees smoothly with \(N_a\) across \(\beta_b=0\). The case \(\beta_a(x_0)=0\) follows by interchanging \(a\) and \(b\). Thus \(\zeta_\lambda\) is smooth near \(x_0\), completing the proof.
\end{proof}

We next estimate the differential of \(\zeta_\lambda\) in terms of the mean curvature of \(\Sigma_\lambda\). For a linear map \(B\) between finite-dimensional inner-product spaces, let \(\|B\|_{\operatorname{tr}}\) denote the sum of its singular values, equivalently,
\[
\|B\|_{\operatorname{tr}}=\operatorname{tr}\left((B^*B)^{1/2}\right).
\]

\begin{lemma}\label{lem:zeta-estimate}
Let \(x\in\Sigma_\lambda\) satisfy \(I_\lambda(x)=\{a,b\}\) and \(d_g(x,K_P)>C_0\lambda^{-1}\). With \(L\), \(e\), and \(\kappa\)  as in Subsection~\ref{subsec:codimension-two-normals}, one has
\[
|d\zeta_\lambda(e)|\leq\kappa+C,\qquad |d\zeta_\lambda(Y)|\leq C
\]
for every \(g\)-unit vector \(Y\in L\). Consequently,
\begin{equation}\label{eq:zeta-estimate}
    \|d\zeta_\lambda\|_{\operatorname{tr}}\leq H_{\Sigma_\lambda}^g+C
\end{equation}
throughout the region \(d_g(x,K_P)>C_0\lambda^{-1}\).
\end{lemma}
\begin{proof}
    At a point where \(I_\lambda(x)=\{a,b\}\), the curve \(t\mapsto\Gamma(N_a,N_b,t)\) has constant speed \(\alpha_{ab}\). Hence, for \(X\in T_x\Sigma_\lambda\),
 \[
 \begin{aligned}
     |d\zeta_\lambda(X)|=&\alpha_{ab}\left|d\left(\frac{s}{\alpha}\right)(X)\right|\\
     \leq&\frac{\alpha_{ab}}{\alpha}|ds(X)|+\frac{\alpha_{ab}s}{\alpha^2}|d\alpha(X)|.
 \end{aligned}
 \]
 Since \(\alpha_{ab}\leq\alpha\), \(0\leq s\leq\alpha\), we obtain
 \[
 |d\zeta_\lambda(X)|\leq |ds(X)|+C|X|_g.
 \]
 The first two estimates then follow from Lemma~\ref{lem:s-estimates}.

 Choose a \(g\)-orthonormal basis \(Y_1,\ldots,Y_{n-2}\) of \(L\). Then
 \[
 \|d\zeta_\lambda\|_{\operatorname{tr}}\leq|d\zeta_\lambda(e)|+\sum_{j=1}^{n-2}|d\zeta_\lambda(Y_j)|\leq\kappa+C.
 \]
 By \eqref{eq:second-form-estimates}, \(H_{\Sigma_\lambda}^g=\kappa+O(1)\), which proves \eqref{eq:zeta-estimate} when \(I_\lambda(x)=\{a,b\}\).

 If \(I_\lambda(x)=\{a\}\), then \(\zeta_\lambda=N_a\) and \(d\zeta_\lambda=0\), while \(H_{\Sigma_\lambda}^g=H_{F_a}^g\geq 0\). Thus \eqref{eq:zeta-estimate} also holds in this case.
\end{proof}

Although \(\zeta_\lambda\) may be defined near \(K_P\), the preceding argument gives no control of its differential, so near \(K_P\) we replace it by the map
\[
q(x)=\frac{x-p_0}{|x-p_0|},\qquad x\in P\setminus\{p_0\}.
\]
The precise form of \(q\) is not essential; the key point is the estimate in the following lemma.

\begin{lemma}\label{lem:q-degree}
For all sufficiently large \(\lambda\), the restriction \(q|_{\Sigma_\lambda}\) has degree one. Moreover, there is \(c_0>0\), independent of \(\lambda\), such that
\begin{equation}\label{eq:hemisphere}
    \langle q(x), N_a\rangle\geq c_0
\end{equation}
whenever \(x\in\Sigma_\lambda\) and \(a\in I_\lambda(x)\).
\end{lemma}
\begin{proof}
    As in Lemma~\ref{lem:smoothing}, every ray from \(p_0\) meets \(\Sigma_\lambda\) in exactly one point. Hence \(q|_{\Sigma_\lambda}\) has degree one.

    Set \(\delta:=\min_a\left(-u_a(p_0)\right)>0\). If \(a\in I_\lambda(x)\), then \(u_a(x)>-2/\lambda\) and therefore
    \[
    \langle q(x), N_a\rangle=\frac{u_a(x)-u_a(p_0)}{|x-p_0|}\geq\frac{\delta-2/\lambda}{|x-p_0|}.
    \]
    Since \(|x-p_0|\) is uniformly bounded on \(P\), this proves \eqref{eq:hemisphere} for all sufficiently large \(\lambda\).
\end{proof}

We next make a smooth transition between \(q\) and \(\zeta_\lambda\). By a standard Whitney decomposition argument, there is a function \(r\), smooth away from \(K_P\) such that on a neighborhood of \(P\),
\begin{equation}\label{eq:regularized-distance}
C^{-1}d_g(x,K_P)\leq r(x)\leq Cd_g(x,K_P),\qquad |dr(x)|_g\leq C.
\end{equation}

Set \(\delta_\lambda=\lambda^{-1/4}\), so that \(\lambda^{-1}=\delta_\lambda^4=o(\delta_\lambda^2)\). Choose a smooth log cutoff function \(\tau_\lambda:[0,\infty)\to[0,1 ]\) such that
\[
\tau_\lambda(t)=0\quad (t\leq\delta_\lambda^2),\qquad\tau_\lambda(t)=1\quad (t\geq\delta_\lambda),\qquad |\tau_\lambda^\prime(t)|\leq \frac{C}{t|\log\delta_\lambda|}.
\]
Note that \(\tau_\lambda(r)\) is smooth across \(K_P\).

On \(\{r>\delta_\lambda^2\}\), \eqref{eq:regularized-distance} gives \(d_g(x,K_P)\geq c\delta_\lambda^2\). Since \(\lambda^{-1}=o(\delta_\lambda^2)\), Lemma~\ref{lem:zeta-estimate} applies for all sufficiently large \(\lambda\). Moreover \eqref{eq:hemisphere}, \eqref{eq:spherical-interpolation} and the concavity of \(\sin\) on \([0,\pi]\) give
\begin{equation}\label{eq:q-zeta-angle}
    \langle q,\zeta_\lambda\rangle\geq c_0.
\end{equation}
Thus \(q(x)\neq-\zeta_\lambda(x)\). So we can define
\begin{equation}\label{eq:eta-definition}
\eta_\lambda(x)=\begin{cases}
    q(x), &r(x)\leq\delta_\lambda^2,\\[2mm]
    \Gamma\left(q(x),\zeta_\lambda(x),\tau_\lambda(r(x))\right), &r(x)>\delta_\lambda^2.
\end{cases}
\end{equation}

\begin{lemma}\label{lem:eta}
For all sufficiently large \(\lambda\), \(\eta_\lambda:\Sigma_\lambda\to\mathbb S^{n-1}\) is smooth and has degree one. Moreover, \(\eta_\lambda=\zeta_\lambda\) on \(\{r\geq\delta_\lambda\}\). For every \(a\) and every \(V\Subset F_a^\circ\), there exists \(\lambda_V>0\) such that \(\eta_\lambda=N_a\) on \(V\) for all \(\lambda\geq\lambda_V\).
\end{lemma}
\begin{proof}
   The smoothness of \(\eta_\lambda\) follows directly from \eqref{eq:q-zeta-angle}, \eqref{eq:eta-definition} and the choice of \(\tau_\lambda\). For \(0\leq\sigma\leq 1\), define
   \[
   H_\sigma(x)=\begin{cases}
       q(x), &r(x)\leq\delta_\lambda^2,\qquad\\[2mm]
       \Gamma\left(q(x),\zeta_\lambda(x),\sigma\tau_\lambda(r(x))\right), &r(x)>\delta_\lambda^2.
   \end{cases}
   \]
   Then \(H_0=q|_{\Sigma_\lambda}\) and \(H_1=\eta_\lambda\). Thus \(\operatorname{deg }\eta_\lambda=\operatorname{deg}(q|_{\Sigma_\lambda})=1\) by Lemma~\ref{lem:q-degree}.

If \(r\geq\delta_\lambda\), then \(\tau_{\lambda}(r)=1\), and hence \(\eta_\lambda=\zeta_\lambda\). Finally, for fixed \(a\) and \(V\Subset F_a^\circ\), Lemma~\ref{lem:smoothing} and \(\delta_\lambda\to 0\) give \(\lambda_V>0\) such that \(I_\lambda(x)=\{a\}\) and \(r(x)\geq\delta_\lambda\) for every \(x\in V\) and \(\lambda\geq\lambda_V\). Hence \(\eta_\lambda=\zeta_\lambda=N_a\) on \(V\).
\end{proof}

To estimate \(d\eta_\lambda\), we use the following elementary property of \(\Gamma\).
\begin{lemma}\label{lem:spherical-contraction}
For \(U\in\mathbb S^{n-1}\) and \(0\leq\tau\leq 1\), the map
\[
N\mapsto\Gamma(U,N,\tau)
\]
is \(1\)-Lipschitz from \(\{N\in\mathbb S^{n-1}:\langle U, N\rangle>0\}\) to itself with respect to the spherical metric.
\end{lemma}
\begin{proof}
    In polar coordinates centered at \(U\), the spherical metric is 
    \[
    d\vartheta^2+\sin^2\vartheta g_{\mathbb S^{n-2}},\qquad 0\leq\vartheta<\frac{\pi}{2}.
    \]
    In these coordinates, the map sends \((\vartheta,\omega)\) to \((\tau\vartheta,\omega)\), and is therefore \(1\)-Lipschitz.
\end{proof}

Define 
\begin{equation}\label{eq:weight-definition}
w_\lambda=\frac{\mathbf 1_{\{\delta_\lambda^2<r<\delta_\lambda\}}}{r|\log\delta_\lambda|}\qquad\text{on }\Sigma_\lambda.
\end{equation}
Here \(\mathbf 1_E\) denotes the characteristic function of \(E\). For \(t\in\mathbb R\), write \(t_+=\max\{t,0\}\).

\begin{lemma}\label{lem:eta-estimate}
There exists \(C>0\) such that, for all sufficiently large \(\lambda\),
\begin{equation}\label{eq:eta-derivative}
\|d\eta_\lambda\|_{\operatorname{tr}}\leq\begin{cases}
    C, &r\leq\delta_\lambda^2,\\[1mm]
    \|d\zeta_\lambda\|_{\operatorname{tr}}+C(1+w_\lambda), &r>\delta_\lambda^2.
\end{cases}
\end{equation}
Consequently, \eqref{eq:zeta-estimate} gives
\begin{equation}\label{eq:eta-curvature}
\left(\|d\eta_\lambda\|_{\operatorname{tr}}-H_{\Sigma_\lambda}^g\right)_+\leq C(1+w_\lambda)\qquad\text{on }\{r>\delta_\lambda^2\}.
\end{equation}
\end{lemma}
\begin{proof}
    On \(\{r\leq\delta_\lambda^2\}\), one has \(\eta_\lambda=q\), so \(\|d\eta_\lambda\|_{\operatorname{tr}}=\|dq\|_{\operatorname{tr}}\leq C\).

    Suppose that \(r>\delta_\lambda^2\). By \eqref{eq:q-zeta-angle}, the triples \(\left(q,\zeta_\lambda,\tau_\lambda(r)\right)\) lie in a fixed compact subset of the domain of \(\Gamma\). The smoothness of \(\Gamma\) therefore gives
    \[
    \|\partial_U\Gamma\circ dq\|_{\operatorname{tr}}\leq C,\qquad \|\partial_\tau\Gamma\otimes d(\tau_\lambda(r))\|_{\operatorname{tr}}\leq C|d(\tau_\lambda(r))|_g\leq Cw_\lambda.
    \]
    Moreover, for fixed \(U\) and \(\tau\), Lemma~\ref{lem:spherical-contraction} gives \(|\partial_N\Gamma[\xi]|\leq|\xi|\) for every tangent vector \(\xi\), and consequently
\[
\|\partial_N\Gamma\circ d\zeta_\lambda\|_{\operatorname{tr}}\leq\|d\zeta_\lambda\|_{\operatorname{tr}}.
\]
Applying these estimates to 
\[
d\eta_\lambda=\partial_U\Gamma\circ dq+\partial_N\Gamma\circ d\zeta_\lambda+\partial_\tau\Gamma\otimes d(\tau_\lambda(r))
\]
gives \eqref{eq:eta-derivative}. Estimate \eqref{eq:eta-curvature} then follows from \eqref{eq:zeta-estimate}.
\end{proof}

\section{Proof of the main theorem}

We use the Dirac boundary formulation from \cite[Section~2]{Brendle}; see also \cite[Section~2]{BaerBrendleChowHanke}. We begin with the corresponding boundary estimate.

\subsection{The Dirac boundary estimate}

We denote by \(\mathrm{Cl}(\mathbb R^n)\) the complex Clifford algebra generated by \(\mathbb R^n\) subject to the relations
\[
vw+wv=-2\langle v,w\rangle,\qquad v,w\in\mathbb R^n.
\]
Let \((\Delta_n,\omega)\) be a Hermitian \(\mathrm{Cl}(\mathbb R^n)\)-module. Thus \(\Delta_n\) is a complex Hermitian vector space and \(\omega: \mathrm{Cl}(\mathbb R^n)\to\operatorname{End}_{\mathbb C}(\Delta_n)\) is a complex algebra homomorphism such that \(\omega(v)^*=-\omega(v)\) for every \(v\in\mathbb R^n\), where \(^{*}\) denotes the adjoint with respect to the Hermitian inner product on \(\Delta_n\).

Let \((Y^n, g)\) be a compact oriented Riemannian spin manifold with smooth boundary \(\Sigma\). Let \(S_Y\) be its spinor bundle, with Clifford multiplication \(c\). Viewing \(\Delta_n\) as a trivial bundle on \(Y\), the spin connection induces a connection on \(\operatorname{Hom}(\Delta_n,S_Y)\cong S_Y\otimes\Delta_n^*\). The corresponding Dirac operator is 
\[
DA=\sum_{j=1}^nc(e_j)\nabla_{e_j}A.
\]
Let \(\nu\) be the outward unit normal of \(\Sigma=\partial Y\). For a local orthonormal frame \(e_1,\ldots,e_{n-1}\) of \(T\Sigma\), define
\[
\mathcal D^\Sigma A=\sum_{j=1}^{n-1 }c(\nu)c(e_j)\nabla_{e_j}A+\frac{1}{2}H_\Sigma^gA.
\]
Given a smooth map \(\eta:\Sigma\to\mathbb S^{n-1}\), set
\[
\chi_\eta A=-c(\nu)A\omega(\eta).
\]

\begin{proposition}\label{prop:dirac-boundary-estimate}
Suppose that a smooth section \(A\) of \(\operatorname{Hom}(\Delta_n,S_Y)\) satisfies
\[
DA=0\quad\text{in }Y,\qquad \chi_\eta A=A\quad\text{on }\Sigma.
\]
Then
\begin{equation}\label{eq:dirac-boundary-estimate}
\int_Y|\nabla A|^2+\frac{1}{4}\int_YR_g|A|^2\leq\frac{1}{2}\int_\Sigma\left(\|d\eta\|_{\operatorname{tr}}-H_{\Sigma}^g\right)_+|A|^2.
\end{equation}
\end{proposition}
\begin{proof}
    For \(x\in\Sigma\) and \(U,V\in\operatorname{Hom}(\Delta_n,S_{Y,x})\), the induced Hermitian inner product is \(\langle U, V\rangle=\operatorname{tr}(U^*V)\). Since \(c(\nu)^*=-c(\nu)\) and \(\omega(\eta)^*=-\omega(\eta)\), we have
    \[
    \langle\chi_\eta U, V\rangle=-\operatorname{tr}\left(\omega(\eta)U^*c(\nu)V\right)=-\operatorname{tr}\left(U^*c(\nu)V\omega(\eta)\right)=\langle U, \chi_\eta V\rangle.
    \]
Thus \(\chi_\eta^*=\chi_\eta\).

Let \(e_1,\ldots,e_{n-1}\) be a local orthonormal frame of \(T\Sigma\). A direct computation, as in \cite[Proposition~2.5]{Brendle}, gives
\[
\chi_\eta\mathcal D^\Sigma A+\mathcal D^\Sigma(\chi_\eta A)=-\sum_{j=1}^{n-1}c(e_j)A\omega\left(d\eta(e_j)\right).
\]
Fix \(x\in\Sigma\). We may choose such a local frame around \(x\) and an orthonormal basis \(E_1,\ldots,E_{n-1}\) of \(T_{\eta(x)}\mathbb S^{n-1}\), so that \(d\eta_x(e_j)=\sigma_jE_j\) with \(\sigma_j\geq 0\). Since Clifford multiplication by a unit vector is unitary, \(|\langle c(e_j)A\omega(E_j),A\rangle|\leq |A|^2\). So at \(x\), we have
\[
\begin{aligned}
    2\operatorname{Re}\langle\mathcal D^\Sigma A, A\rangle=&\operatorname{Re}\langle\chi_\eta\mathcal D^\Sigma A+\mathcal D^\Sigma(\chi_\eta A), A\rangle\\
    =&-\sum_{j=1}^{n-1}\sigma_j\operatorname{Re}\langle c(e_j)A\omega(E_j), A\rangle\leq\sum_{j=1}^{n-1}\sigma_j|A|^2=\|d\eta_x\|_{\operatorname{tr}}|A|^2.
\end{aligned}
\]
Finally, since \(DA=0\), the Schr\"odinger--Lichnerowicz formula gives
\[
\begin{aligned}
    \int_Y|\nabla A|^2+\frac{1}{4}\int_YR_g|A|^2=&\int_\Sigma\left(\operatorname{Re}\langle\mathcal D^\Sigma A, A\rangle-\frac{1}{2}H_{\Sigma}^g|A|^2\right)\\
    \leq&\frac{1}{2}\int_\Sigma\left(\|d\eta\|_{\operatorname{tr}}-H_\Sigma^g\right)_+|A|^2.
\end{aligned}
\]
\end{proof}

\subsection{The boundary error}
\label{subsec:boundary-error}
We now estimate the right-hand side of \eqref{eq:dirac-boundary-estimate} for \((Y,\eta)=(P_\lambda,\eta_\lambda)\). The level-set formula gives
\[
H_{\Sigma_\lambda}^g=\frac{1}{|d\mathcal F_\lambda|_g}\sum_a\left[\lambda^2\Phi''(\lambda u_a)\left(|du_a|_g^2-du_a(\nu_\lambda)^2\right)+\lambda\Phi'(\lambda u_a)\operatorname{tr}_{T\Sigma_\lambda}\nabla_g^2u_a\right].
\]
Since \(\Phi''\geq 0\) and \(\nabla_g^2u_a\) are uniformly bounded, 
\[
H_{\Sigma_\lambda}^g\geq -\frac{C\lambda}{|d\mathcal{F}_\lambda|_g}\sum_a\Phi^\prime(\lambda u_a),
\]
where \(C\) is independent of \(\lambda\). On the other hand, \eqref{eq:normal-combination} and the uniform lower bounds for \(|du_a|_g\) give
\[
|d\mathcal{F}_\lambda|_g=\lambda\left|\sum_a\Phi^\prime(\lambda u_a)du_a\right|_g\geq c\lambda\sum_a\Phi^\prime(\lambda u_a).
\]
Hence,
\begin{equation}\label{eq:mean-curvature-lower-bound}
H_{\Sigma_\lambda}^g\geq -C.
\end{equation}

Set \(W_\lambda=\left(\|d\eta_\lambda\|_{\operatorname{tr}}-H_{\Sigma_\lambda}^g\right)_+\). For \(S\subset P\) and \(\rho>0\), write \(\mathcal N_\rho^g(S)=\{x\in P: d_g(x,S)<\rho\}\). 

\begin{lemma}\label{lem:boundary-error-bound}
With \(w_\lambda\) as in \eqref{eq:weight-definition}, there exists \(C>0\) such that, for all sufficiently large \(\lambda\),
\begin{equation}\label{eq:boundary-error-bound}
W_\lambda\leq C\left(\mathbf 1_{\mathcal N_{C\lambda^{-1}}^g(E_P)}+\mathbf 1_{\mathcal N_{C\delta_\lambda}^g(K_P)}+w_\lambda\right)\qquad\text{on }\Sigma_\lambda.
\end{equation}
\end{lemma}
\begin{proof}
    On \(\{r\leq\delta_\lambda^2\}\), \eqref{eq:eta-derivative} and \eqref{eq:mean-curvature-lower-bound} give \(W_\lambda\leq C\). By \eqref{eq:regularized-distance},
    \[
    \{r\leq\delta_\lambda^2\}\subset\mathcal N_{C\delta_\lambda^2}^g(K_P)\subset\mathcal N_{C\delta_\lambda}^g(K_P).
    \]
On \(\{r>\delta_{\lambda}^2\}\), \eqref{eq:eta-curvature} gives
\[
W_\lambda\leq C(1+w_\lambda).
\]
Moreover, Lemma~\ref{lem:smoothing} and \eqref{eq:regularized-distance}, together with \(\lambda^{-1}=o(\delta_\lambda^2)\), imply that \(|I_\lambda(x)|\leq 2\). If \(|I_\lambda(x)|=2\), Lemma~\ref{lem:smoothing} further gives \(x\in\mathcal N_{C\lambda^{-1}}^g(E_P)\). If \(I_\lambda(x)=\{a\}\) and \(x\notin \mathcal N_{C\delta_\lambda}^g(K_P)\), then \(r(x)\geq\delta_\lambda\). At such a point, \(d\eta_\lambda=0\) and \(H_{\Sigma_\lambda}^g=H_{F_a}^g\geq 0\), so \(W_\lambda=0\). Therefore
\[
\{r>\delta_\lambda^2\}\cap\{W_\lambda>0\}\subset \mathcal N_{C\lambda^{-1}}^g(E_P)\cup\mathcal N_{C\delta_\lambda}^g(K_P).
\]
Together with the preceding estimates, this proves \eqref{eq:boundary-error-bound}.
\end{proof}

For \(P_\lambda\), we have the following uniform extension and trace estimates.

\begin{lemma}\label{lem:uniform-extension}
There is a bounded neighborhood \(\widetilde P\supset P\) with smooth boundary such that, for every smooth Hermitian vector bundle \(\mathcal V\to \widetilde P\) with a compatible connection, the following hold.

There is a \(C>0\) such that, for all sufficiently large \(\lambda\), there is an extension operator
\[
E_\lambda:W^{1,2}(P_\lambda,\mathcal V)\to W^{1,2}(\widetilde P,\mathcal V)
\]
satisfying \(\|E_\lambda\|\leq C\).

For every open ball \(B\Subset P^\circ\), there is \(C>0\) such that for all sufficiently large \(\lambda\),
\begin{equation}\label{eq:uniform-trace-poincare}
\|\psi\|_{L^2(P_\lambda)}^2+\|\psi\|_{L^{\frac{2(n-1)}{n-2}}(\Sigma_\lambda)}^2\leq C\left(\|\nabla\psi\|_{L^2(P_\lambda)}^2+\|\psi\|_{L^2(B)}^2\right)
\end{equation}
for every \(\psi\in W^{1,2}(P_\lambda,\mathcal V)\).
\end{lemma}
\begin{proof}
    Choose a smooth bounded domain \(\widetilde P\) containing \(P\) and compactly contained in the domain of \(g\). On \(\widetilde P\), the metric \(g\) is uniformly equivalent to the Euclidean metric, so we may work with the Euclidean metric. 

    Let \(\Theta_\lambda:B_1\to P_\lambda\) denote the radial parametrization centered at \(p_0\). As observed in the proof of Lemma~\ref{lem:smoothing}, the maps \(\Theta_\lambda\) are uniformly bi-Lipschitz. The extension theorem for Lipschitz domains therefore gives the required extension operators on \(\widetilde P\), with norms bounded independently of \(\lambda\).

    Fix an open ball \(B\Subset P^\circ\). For all sufficiently large \(\lambda\), the sets \(\Theta_\lambda^{-1}(B)\) contain balls of radius bounded below independently of \(\lambda\). The Poincar\'e and trace inequalities on \(B_1\), applied to \(|\psi|\circ\Theta_\lambda\), together with Kato's inequality and the uniform bi-Lipschitz bounds, prove the second statement.
\end{proof}

\begin{lemma}\label{lem:boundary-error-smallness}
Fix an open ball \(B\Subset P^\circ\), and let \(\mathcal V\to \widetilde P\) be a smooth Hermitian vector bundle with a compatible connection. There are numbers \(\varepsilon_\lambda>0\) such that \(\varepsilon_\lambda\to 0\) as \(\lambda\to \infty\) and for all sufficiently large \(\lambda\) and every \(\psi\in W^{1,2}(P_\lambda,\mathcal V)\),
\begin{equation}\label{eq:boundary-error-smallness}
\int_{\Sigma_\lambda}W_\lambda|\psi|^2\leq\varepsilon_\lambda\left(\int_{P_\lambda}|\nabla\psi|^2+\int_B|\psi|^2\right).
\end{equation}
\end{lemma}
\begin{proof}
Since \(E_P\) is a finite union of \((n-2)\)-dimensional faces, Lemma~\ref{lem:smoothing} implies
\(\mathcal H^{n-1}\left(\Sigma_\lambda\cap\mathcal N_{C\lambda^{-1}}^g(E_P) \right)\leq C\lambda^{-1}\). Hence H\"older's inequality and \eqref{eq:uniform-trace-poincare} give
\begin{equation}\label{eq:edge-error-estimate}
\int_{\Sigma_\lambda}\mathbf 1_{\mathcal N_{C\lambda^{-1}}^g(E_P)}|\psi|^2\leq C\lambda^{-\frac{1}{n-1}}\left(\int_{P_\lambda}|\nabla\psi|^2+\int_B|\psi|^2\right).
\end{equation}
Near \(K_P\), let \(x\in F_a\) satisfy \(d(x,K_P)<\rho_0\), where \(\rho_0>0\) is smaller than the distance between any two disjoint faces. Choose \(J\) with \(|J|=3\) such that \(d(x, F_J)=d(x,K_P)\). Then \(F_J\cap F_a\neq\emptyset\). Choose \(I\subset J\cup\{a\}\) with \(|I|=3\) and \(a\in I\). Then \(F_I\neq\emptyset\), and \eqref{eq:face-distance} gives
\[
d(x,F_I)\leq d(x,F_{J\cup\{a\}})\leq C\max_{j\in J}|u_j(x)|\leq Cd(x, F_J)=Cd(x,K_P).
\]
Since \(F_I\subset K_P\) whenever \(|I|=3\), it follows that
\begin{equation}\label{eq:triple-face-distance}
d(x,K_P)\leq\min_{\substack{|I|=3,a\in I\\ F_I\neq\emptyset}}d(x,F_I)\leq Cd(x,K_P).
\end{equation}
For \(q\in \partial P\) and \(s\in[0,1]\), set 
\[
T_\lambda\left(p_0+s(q-p_0)\right)=p_0+s\left(\Pi_\lambda(q)-p_0\right).
\]
By Lemma~\ref{lem:smoothing}, \(T_\lambda:P\to P_\lambda\) is uniformly bi-Lipschitz and \(\sup_{x\in P}d_g(T_\lambda(x),x)\leq C\lambda^{-1}\). Thus if \(T_\lambda(x)\in\mathcal N_{C\delta_\lambda}^g(K_P)\), then \(d(x,K_P)< C\delta_\lambda\). On the support of \(w_\lambda\circ T_\lambda\), \eqref{eq:regularized-distance} and \(\lambda^{-1}=o(\delta_\lambda^2)\) give
\[
c\delta_\lambda^2<d(x,K_P)<C\delta_\lambda,\qquad cd(x,K_P)\leq r(T_\lambda(x))\leq Cd(x,K_P).
\]
Combining these with \eqref{eq:triple-face-distance}, we obtain for \(x\in F_a\),
\[
\left(\mathbf 1_{\mathcal N_{C\delta_\lambda}^g(K_P)}+w_\lambda\right)(T_\lambda(x))\leq C\sum_{\substack{|I|=3,a\in I\\F_I\neq\emptyset}}\left(\mathbf 1_{\{d(x,F_I)<C\delta_\lambda\}}+\frac{\mathbf 1_{\{c\delta_\lambda^2<d(x,F_I)<C\delta_\lambda\}}}{d(x,F_I)|\log\delta_\lambda|}\right).
\]
Fix \(I\) in the sum on the right. Note that \(F_I\subset F_a\) and \(\operatorname{codim }F_I\geq 3\). Write \(\mathbb R^n=\mathbb R^{n-3}\times\mathbb R^2\times \mathbb R\), with coordinates \((y,z,t)\), so that \(F_a\subset\{t=0\}\) and \(F_I\subset\{z=t=0\}\). Then
\[
d\left((y,z,0), F_I\right)^2=d\left((y,0,0), F_I\right)^2+|z|^2.
\]
Of the two terms indexed by \(I\), the first has \(L_z^2\)-norm at most \(C\delta_\lambda\), uniformly in \(y\). For the second, one has
\[
\frac{1}{|\log\delta_\lambda|^2}\int_{c\delta_\lambda^2<d((y,z,0),F_I)<C\delta_\lambda}\frac{dz}{d((y,z,0),F_I)^2}\leq \frac{C}{|\log\delta_\lambda|^2}\int_{c\delta_\lambda^2}^{C\delta_\lambda}\frac{ds}{s}\leq\frac{C}{|\log\delta_\lambda|}.
\]

By \eqref{eq:uniform-trace-poincare} and the uniform bi-Lipschitz bounds for \(T_\lambda\), \(|\psi|\circ T_\lambda\) has an extension \(v\in W^{1,2}(\mathbb R^n)\) satisfying
\[
\|v\|_{W^{1,2}(\mathbb R^n)}^2\leq C\left(\|\nabla\psi\|_{L^2(P_\lambda)}^2+\|\psi\|_{L^2(B)}^2\right).
\]
For almost every \(y\), the trace inequality gives 
\[
\|v(y,\cdot,0)\|_{L^4(\mathbb R^2)}^2\leq C\|v(y,\cdot,\cdot)\|_{W^{1,2}(\mathbb R^3)}^2.
\]
It follows that 
\begin{equation}\label{eq:higher-codimension-error-estimate}
\begin{aligned}
    \int_{\Sigma_\lambda}\left(\mathbf 1_{\mathcal N_{C\delta_\lambda}^g(K_P)}+w_\lambda\right)|\psi|^2\leq&C\left(\delta_\lambda+|\log\delta_\lambda|^{-\frac{1}{2}}\right)\int_{\mathbb R^{n-3}}\|v(y,\cdot,0)\|_{L^4(\mathbb R^2)}^2dy\\
    \leq&C\left(\delta_\lambda+|\log\delta_\lambda|^{-\frac{1}{2}}\right)\|v\|_{W^{1,2}(\mathbb R^n)}^2\\
    \leq&C\left(\delta_\lambda+|\log\delta_\lambda|^{-\frac{1}{2}}\right)\left(\|\nabla\psi\|_{L^2(P_\lambda)}^2+\|\psi\|_{L^2(B)}^2\right).
\end{aligned}
\end{equation}
Combining \eqref{eq:boundary-error-bound}, \eqref{eq:edge-error-estimate}, \eqref{eq:higher-codimension-error-estimate}, we obtain
\[
\int_{\Sigma_\lambda}W_\lambda|\psi|^2\leq C\left(\lambda^{-\frac{1}{n-1}}+\delta_\lambda+|\log\delta_\lambda|^{-\frac{1}{2}}\right)\left(\int_{P_\lambda}|\nabla\psi|^2+\int_B|\psi|^2\right).
\]
Since \(\delta_\lambda=\lambda^{-1/4}\), this proves \eqref{eq:boundary-error-smallness}.
\end{proof}

\subsection{Passage to the limit and rigidity}

Assume first that \(n\) is odd, and choose an irreducible Hermitian \(\mathrm{Cl}(\mathbb R^n)\)-module \((\Delta_n,\omega)\) such that
\[
i^{\frac{n+1}{2}}\omega(E_1)\cdots\omega(E_n)=I
\]
for every positively oriented orthonormal basis \((E_1,\ldots,E_n)\) of \(\mathbb R^n\). We use the same convention for Clifford multiplication on the spinor bundle \(S_Y\) of \((Y,g)\), so that 
\[
i^{\frac{n+1}{2}}c(e_1)\cdots c(e_n)=I
\]
for every positively oriented orthonormal frame \(e_1,\ldots,e_n\) of \(TY\). Here \(I\) denotes the identity endomorphism of \(\Delta_n\) in the first formula and of \(S_Y\) in the second.

We shall use the following existence result.

\begin{proposition}\label{prop:dirac-existence}
Let \(\Omega\subset\mathbb R^n\) be a smooth compact convex domain, and let \(g\) be a smooth metric on a neighborhood of \(\overline{\Omega}\). If \(\eta:\partial\Omega\to\mathbb S^{n-1}\) is smooth and \(\operatorname{deg }\eta=1\), then there exists a nonzero smooth section \(A\) of \(\operatorname{Hom}(\Delta_n,S_\Omega)\) satisfying
\[
DA=0\quad\text{in }\Omega,\qquad \chi_\eta A=A\quad\text{on }\partial\Omega.
\]
\end{proposition}
\begin{proof}
    Note that \(\partial\Omega\) is diffeomorphic to \(\mathbb S^{n-1}\) and that \(\eta\) and the Gauss map of \(\partial\Omega\) both have degree one. Hence these maps are homotopic. By \cite[Proposition~2.15]{Brendle}, the boundary problem has positive Fredholm index and hence a nontrivial kernel.
\end{proof}

We now pass to the limit along the smooth approximations \(P_\lambda\).
\begin{proposition}\label{prop:parallel-limit}
There exists a nonzero parallel section \(A\) of \(\operatorname{Hom}(\Delta_n,S_{P^\circ})\) which extends smoothly to a neighborhood of \(P\) and satisfies 
\begin{equation}\label{eq:boundary-relation}
c(\nu_a)A=A\omega(N_a)\qquad\text{on }F_a
\end{equation}
for every \(a\).
\end{proposition}
\begin{proof}
Let \(\widetilde P\) be as in Lemma~\ref{lem:uniform-extension}, and let \(S_{\widetilde P}\) be the spinor bundle of \((\widetilde P, g)\) and \(S_{P_\lambda}=S_{\widetilde P}|_{P_\lambda}\). By Lemma~\ref{lem:smoothing} and Lemma~\ref{lem:eta}, \(P_\lambda\) is a smooth compact convex domain and \(\operatorname{deg }\eta_\lambda=1\). Proposition~\ref{prop:dirac-existence} therefore gives, after rescaling, a nonzero smooth section \(A_\lambda\) of \(\operatorname{Hom}(\Delta_n,S_{P_\lambda})\) satisfying
\begin{equation}\label{eq:A-lambda}
    DA_\lambda=0\quad\text{in } P_\lambda,\qquad \chi_{\eta_\lambda}A_\lambda=A_\lambda\quad\text{on }\Sigma_\lambda, \qquad\int_{P_\lambda}|A_\lambda|^2=1.
\end{equation}
Fix an open ball \(B\Subset P^\circ\). For all sufficiently large \(\lambda\), one has \(B\subset P_\lambda\), and Proposition~\ref{prop:dirac-boundary-estimate} and Lemma~\ref{lem:boundary-error-smallness}  give
\[
\int_{P_\lambda}|\nabla A_\lambda|^2+\frac{1}{4}\int_{P_\lambda}R_g|A_\lambda|^2\leq\frac{\varepsilon_\lambda}{2}\left(\int_{P_\lambda}|\nabla A_\lambda|^2+\int_B|A_\lambda|^2\right).
\]
Since \(R_g\geq 0\) and by \eqref{eq:A-lambda}, \(\int_B|A_\lambda|^2\leq 1\),
\begin{equation}\label{eq:A-energy}
\int_{P_\lambda}|\nabla A_\lambda|^2\leq\frac{\varepsilon_\lambda}{2-\varepsilon_\lambda}\longrightarrow 0.
\end{equation}
Lemma~\ref{lem:uniform-extension} provides extensions of \(A_\lambda\) to \(\widetilde P\) that are bounded in \(W^{1,2}(\widetilde P)\). After passing to a subsequence,
\begin{equation}\label{eq:A-convergence}
A_\lambda\rightharpoonup A\quad\text{in }W^{1,2}(\widetilde P),\qquad A_\lambda\to A\quad\text{in }L^2(\widetilde P).
\end{equation}
Every compact subset of \(P^\circ\) is contained in \(P_\lambda\) for all sufficiently large \(\lambda\). Hence \eqref{eq:A-energy} implies that \(\nabla A=0\) on \(P^\circ\), so \(A\) is smooth and parallel there.

Since \(P_\lambda\subset P\) and, by Lemma~\ref{lem:smoothing}, \(d_H^g(\Sigma_\lambda,\partial P)\to 0\), the strong \(L^2\) convergence and \eqref{eq:A-lambda} give
\[
\int_P|A|^2=\lim_{\lambda\to\infty}\int_{P_\lambda}|A_\lambda|^2=1,
\]
so \(A\neq 0\). 

Fix \(a\) and \(V\Subset F_a^\circ\). By Lemma~\ref{lem:smoothing} and Lemma~\ref{lem:eta}, for all sufficiently large \(\lambda\), \(\Sigma_\lambda\) coincides with \(F_a\) near \(\overline{V}\) and \(\eta_\lambda=N_a\) on \(V\). Hence \eqref{eq:A-lambda} gives \(-c(\nu_a)A_\lambda\omega(N_a)=A_\lambda\), or equivalently, \(c(\nu_a)A_\lambda=A_\lambda\omega(N_a)\) on \(V\). The compactness of the trace embedding, together with \eqref{eq:A-convergence}, gives \(c(\nu_a)A=A\omega(N_a)\) on \(V\). Fix \(p_0\in P^\circ\). Parallel transport of \(A(p_0)\) along the line segments from \(p_0\) extends \(A\) smoothly to a neighborhood of \(P\). Hence \eqref{eq:boundary-relation} holds on \(F_a\) by continuity.
\end{proof}

\begin{proof}[Proof of Theorem~\ref{thm:main}]
    Let \(A\) be the section given by Proposition~\ref{prop:parallel-limit}. Since \(A\) is parallel, \(A^*A\in\operatorname{End}(\Delta_n)\) is constant. Taking the adjoint of \eqref{eq:boundary-relation} gives
    \[
    A^*c(\nu_a)=\omega(N_a)A^*\qquad\text{on }F_a.
    \]
Together with \eqref{eq:boundary-relation}, this gives
\[
\omega(N_a)A^*A=A^*A\omega(N_a)\qquad\text{for every }a.
\]
Since \(P\) is compact and convex, \(N_1,\ldots, N_m\) span \(\mathbb R^n\). Hence \(A^*A\) commutes with every \(\omega(v)\). By irreducibility and Schur's lemma, \(A^*A\) is a positive scalar multiple of \(I\). After rescaling, \(A^*A=I\). As \(\Delta_n\) and \(S_{P^\circ}\) have the same rank, \(A\) is unitary. Thus the image under \(A\) of an orthonormal basis of \(\Delta_n\) is a parallel orthonormal frame of \(S_{P^\circ}\). Thus the spin connection is flat. For a local orthonormal frame \(e_1,\ldots,e_n\), its curvature satisfies
\[
0=\frac{1}{4}\sum_{i,j=1}^n\langle\operatorname{Rm}_g(X,Y)e_i,e_j\rangle c(e_i)c(e_j).
\]
The spin representation of \(\mathfrak{so}(n)\) is faithful for \(n\geq 3\). Hence \(\operatorname{Rm}_g=0\).

Since the extension of \(A\) is smooth and \(\nabla A=0\) on \(P^\circ\), continuity gives \(\nabla A=0\) on \(P\). For \(Y\in \Gamma(TF_a^\circ)\), differentiating \eqref{eq:boundary-relation} and using that \(N_a\) is constant gives \(c(\nabla_Y^g\nu_a)A=0\). Since \(A\) is unitary, \(c(\nabla_Y^g\nu_a)=0\). The Clifford relation \(c(v)^2=-|v|_g^2I\) gives \(\nabla_Y^g\nu_a=0\). Thus \(F_a\) is totally geodesic.

Suppose that \(F_a\cap F_b\neq\emptyset\). Using \eqref{eq:boundary-relation}, we obtain on \(F_a\cap F_b\)
\[
\begin{aligned}
    -2\langle\nu_a,\nu_b\rangle_gI=&c(\nu_a)c(\nu_b)+c(\nu_b)c(\nu_a)\\
    =&A\left(\omega(N_a)\omega(N_b)+\omega(N_b)\omega(N_a)\right)A^{-1}=-2\langle N_a,N_b\rangle I.
\end{aligned}
\]
Therefore
\[
\langle\nu_a,\nu_b\rangle_g=\langle N_a,N_b\rangle\qquad\text{on }F_a\cap F_b.
\]
This proves the theorem when \(n\) is odd. For \(n\) even, consider \((P\times[-1,1],g+dt^2)\). The odd-dimensional case applies, and restricting its conclusions to \(P\times\{0\}\) proves the theorem.
\end{proof}

\end{document}